\documentclass[twoside,leqno,twocolumn]{article}
\usepackage{ltexpprt}
\usepackage{latexsym, psfrag, epsfig, amsmath,amssymb,bm}

\def\ind{{\mathbf{1}}}
\def\EE{{\mathbb E}}
\def\PP{{\mathbb P}}
\def\NN{{\mathbb N}}
\def\dist{\text{dist}}
\def\diam{\text{diam}}
\def\flood{\text{flood}}
\def\cG{\mathcal{G}}

\def\tx{{\rm tx}}
\def\Bin{{\rm Bin}}

\def\N{{\mathbb N}}

\def\cF{{\mathcal F}}
\def\cE{{\mathcal E}}
\def\hd{\hat{d}}
\def\hS{\hat{S}}
\def\upi{\underline{\pi}}
\def\um{\underline{m}}

\def\uD{\underline{D}}
\def\uphi{\underline{\phi}}
\def\bpi{\bar{\pi}}
\def\bm{\bar{m}}

\def\bD{\bar{D}}

\def\bnu{\bar{\nu}}

\def\hD{\hat{D}}
\def\hS{\hat{S}}

\def\hSigma{\hat{\Sigma}}
\def\dmin{d_{\min}}

\newtheorem{condition}{Condition}

\begin{document}

\title{\Large Flooding in Weighted Random Graphs}
\author{Hamed Amini\thanks{INRIA-ENS, Hamed.Amini@ens.fr} \\
\and
Moez Draief\thanks{Imperial College London, M.Draief@imperial.ac.uk} \\
\and
Marc Lelarge\footnote{INRIA-ENS, Marc.Lelarge@ens.fr}
}
\date{}

\maketitle

\begin{abstract}
In this paper, we study the impact of edge weights on distances in
diluted random graphs. We interpret these weights as delays, and take
them as i.i.d exponential random variables. We analyze the weighted
flooding time defined as the minimum time needed to reach all nodes
from one uniformly chosen node, and the weighted diameter corresponding to
the largest distance between any pair of vertices. Under some regularity conditions on
the degree sequence of the random graph, we show that these quantities
grow as the logarithm of $n$, when the size of the graph $n$ tends to
infinity. We also derive the exact value for the prefactors.

These allow us to analyze an asynchronous randomized broadcast
algorithm for random regular graphs. Our results show that the
asynchronous version of the algorithm performs better than its
synchronized version: in the large size limit of the graph, it will
reach the whole network faster even if the local dynamics are similar
on average.
\end{abstract}

\section{Introduction}

Driven by the distributed nature of modern network architectures,
there has been intense research to devise algorithms to ensure
effective network computation. Of particular interest is the problem
of global node outreach, whereby some major event happening in one
part of the network has to be communicated to all other nodes. In this
context, gossip protocols have been identified as simple, efficient
and robust mechanisms for disseminating and retrieving information for
various network topologies. These mechanisms rely on simple periodic
local operations between neighboring nodes \cite{Pe00}.

Flooding corresponds to the most commonly used such process: a source
node that first records the event notifies all the nodes within its
reach. Subsequently, each of these neighbors forwards information to
all of its neighbors and so on. If the underlying network is connected
such information will eventually reach all the nodes. The performance
of this procedure can be evaluated in terms of the time it takes to
complete. This in particular depends on the underlying network
topology, namely the existence of short paths between different
vertices of the network. In practice, one may imagine that there are
other parameters, besides the network topology, to be taken into
account such as the communication delays between nodes due for example
to congestion. In this context, the spread of the information in the
network can be thought of as a fluid penetrating the network
reminiscent of the problem of first-passage percolation in a random
medium. In this paper, we will consider an asynchronous model in which each edge of the network is equipped with a random delay modeled by an exponential random variable with mean one.

One of the main motivations of our work comes from peer-to-peer
networks. In particular, to motivate our random graph model, we recall
that the most relevant properties of peer-to-peer networks are
connectivity, small average degree, and approximate regularity of the
degrees of the vertices. The random graph model considered in this
paper, explained in detail in Section~\ref{se:mod-res}, has these
properties, and covers the classical $\cG(n, r)$ model, which is the
random graph model in which a graph is drawn uniformly at random from
the set of $n$-vertex $r$-regular graphs, where $r$ is a constant not
depending on $n$. For this model of networks, we consider the push
model for disseminating information: initially, one of the nodes
obtains some piece of information. Then every node which has the
information passes it to another node chosen among its neighbors. The
classical model goes iteratively and all nodes have the same clock. In
each successive round, the nodes having the information choose
independently and uniformly at random the neighbor they transmit
to. In this paper, we analyze an asynchronous randomized broadcast algorithm. Namely nodes are not anymore assumed to be synchronized, so that each node has an independent Poisson clock. A node receiving the information will transmit it to a random neighbor at each tick of its own clock. When the graph of neighbors is a random $r$-regular graph, we show that the asynchronous version of the algorithm performs better than its synchronized version (see Section \ref{se:br}).
To the best of our knowledge, our work is the first to study this
model in an asynchronous version.

From a more theoretical point of view, our work contributes to the
general theory of random graphs by providing new results for the
weighted diameter of connected random graphs.
The analysis of the asymptotic of distances in edge weighted graphs
has received much interest. In \cite{janson99}, Janson considered the
special case of the complete graph with fairly general i.i.d.
weights on edges, including the exponential distribution with
parameter $1$. It is shown that, when $n$ goes to infinity, the
asymptotic distance is
$\log n / n$ for two given points, that the maximum if one point is
fixed and the other varies is $2\log n / n$, and the maximum over all
pairs of points (i.e. the weighted diameter) is $3\log n / n$.
Janson also derived asymptotic results for the corresponding number of hops or hopcount (the number of edges on the paths with the smallest weight).
More recently, a number of papers provide a detailed analysis of the
scaling behavior of the joint distribution of the first passage
percolation and the corresponding hopcount for the complete
graph. In particular, in \cite{bhamidi08, HHM02}, the authors derive limiting
distributions for the first passage percolation on both the complete
graph and dense Erd\"{o}s-R\'enyi random graphs with exponential and
uniform i.i.d.\ weights on edges.
More closely related to the present work, Bhamidi, van der Hofstad and
Hooghiemstra \cite{BHH09} study
first passage percolation on random graphs with finite average degree,
minimum degree greater than $2$ and exponential weights, and derive
explicit distributional asymptotic for the  total weight of the shortest-weight path between two uniformly chosen vertices in
the network. We compare their results to ours in Section \ref{sse:main}. We also explain in Section \ref{sse:proof-bp} why the analysis made in \cite{BHH09} is not sufficient to obtain results for the diameter and how we extend it.

The remainder of the paper is organized as follows.
In Section \ref{se:mod-res}, we define our model for the random graphs and state our main result for the weighted diameter and flooding time. We also compare it to existing works on distances in random graphs. In Section \ref{se:br}, we restrict ourselves to the important class of random regular graphs and analyze a model for asynchronous randomized broadcast. We also compare it to its more classical synchronized version. The main ideas of the proof are given in Section \ref{se:proof}, while technical lemmas are deferred to the Appendix.

\section{Models and Results} \label{se:mod-res}
Given a finite connected graph $G=(V,E)$, the distance $\dist(a,b)$ between two nodes $a$ and $b$  in $V$  is  the number of edges in the shortest path connecting these two vertices. The diameter of  $G$, denoted by $\diam(G)$, is the maximum graph distance between any pair of vertices  in $V$, i.e.
\begin{eqnarray}
\diam(G)=\max\{\dist(a,b), \:a,b\in V \}\: .
\end{eqnarray}
For a graph $G$ with vertex set $V$, the {\em flooding time} is defined by:
$$\flood(G)=\max\{\dist(a,b), \:b\in V \}\:,$$
where $a$ is chosen uniformly at random in $V$.

In this paper, we study the impact of the introduction of edge weights
on the distances in the graph and, in particular, on its diameter.
Such weights can be thought of as economic costs, congestion delays or carrying capabilities that can be encountered in real networks such as transportation systems and communication networks \cite[Chapter 16]{Mieghem}.

For a graph $G=(V,E)$, we assign to each edge $e\in E$
a weight $w_e$.
For any $a,b\in V$, a path between $a$ and $b$ is a sequence
$\pi=(e_1,e_2,\dots e_k)$ where $e_i=(v_{i-1},v_i)\in E$ and $v_i\in V$
for $i\in [1,k]$, with $v_0=a$ and $v_k=b$. We write $e_i\in \pi$ to denote the
fact that edge $e_i$  belongs to the path $\pi$.
For $a,b \in V$, we define $$\text{dist}_w(a,b) = \min_{\pi\in\Pi(a,b)}\sum_{e\in \pi} w_e\:,$$
where $\Pi(a,b)$ denotes the set of all paths from $a$ to $b$ in the
graph. The {\em weighted diameter} and the {\em weighted flooding time} are given by
\begin{eqnarray*}
\diam_w(G) &=& \max\{\dist_w(a,b), \:a,b\in V \}, \ \mbox{ and, } \ \\
\flood_w(G)&=&\max\{\dist_w(a,b), \:b\in V \},
\end{eqnarray*}
where in the definition of flooding, $a$ is chosen uniformly at random in $V$.

Our main results consist of precise asymptotic expressions for the
weighted diameter and weighted flooding time of sparse random graphs on $n$
vertices with (i) a given degree sequence satisfying asymptotic
properties similar to those imposed in \cite{MolReed98}, and (ii)
i.i.d. exponentially distributed weights with parameter $1$.

\subsection{Configuration model.}\label{se:conf}
For $n\in \N$, let $(d_i)_1^n=(d_i^{(n)})_1^n$ be a sequence of
non-negative integers such that $\sum_{i=1}^n d_i$ is even.
For notational simplicity, we will usually not show the dependency on
$n$ explicitly.
By means of the configuration model \cite{bollobas}, we define a random
multigraph with given degree sequence $(d_i)_1^n$, denoted by
$G^*(n,(d_i)_1^n)$ as follows. To each node $i$ we associate $d_i$
labeled half-edges. All half-edges need to be paired to construct the
graph, this is done by a uniform random matching. When a half-edge of $i$ is paired with a half-edge of $j$, we interpret this as an edge between $i$ and $j$. The graph $G^*(n,(d_i)_1^n)$ obtained following this procedure may not be simple, i.e., may contain self-loops due to the pairing of two half-edges of $i$, and multi-edges due to the existence of more than one pairing  between two given nodes. Conditional on the multigraph $G^*(n,(d_i)_1^n)$ being a simple graph, we obtain a uniformly distributed random graph with the
given degree sequence, which we denote by $G(n, (d_i)_1^n)$,
\cite{janson09}.

For $r\in \N$, let $u_r^{(n)}=|\{i, \:d^{(n)}_i=r\}|$ be the number of vertices of degree $r$ and $m^{(n)}$ be the total degree defined by
\begin{eqnarray*}
m^{(n)}=\sum_{i=1}^n d_i = \sum_{r\geq 0} ru_r^{(n)}.
\end{eqnarray*}
From now on, we assume that the sequence $(d_i)_1^n$
satisfies the following regularity conditions analogous to the ones
introduced in \cite{MolReed98} and an additional constraint on the
minimal degree.

\begin{condition}\label{cond-dil}
There exists a distribution $p=\{p_k\}_{k=0}^{\infty}$ such that
\begin{itemize}
\item[(i)] $u_r^{(n)}/n\to p_r$ for every $r\geq 0$ as $n\to
  \infty$;
\item[(ii)] $\lambda:=\sum_{r}rp_r\in (0,\infty)$;
\item[(iii)] $\sum_{i=1}^n d_i^2=O(n)$;
\item[(iv)] for some $\tau > 0$, $\Delta_n:= \max_{i\in V} d_i = O(n^{1/2 -\tau})$;
\item[(v)]  
$\min_{i=1\dots n}d_i=\dmin\geq 3$, and $p_{\dmin} > 0$.
\end{itemize}
\end{condition}

We define $q=\{q_r\}_{r=0}^{\infty}$ the size-biased probability mass function corresponding to $p$, by
\begin{eqnarray}
\label{eq:defq}q_r = \frac{(r+1)p_{r+1}}{\lambda}\: ,
\end{eqnarray}
and let $\nu$ denote its mean, i.e. $\nu = \sum_{r=0}^{\infty} rq_r\in (0,\infty)$.

The condition $\nu > 1$ is equivalent to the existence of a giant component in the configuration model, the size of which is proportional to $n$ (see e.g. \cite{janson09b, MolReed98}). By our Condition \ref{cond-dil} point (v), we have $\nu\geq 2$ and actually, the following lemma, proved in Section~\ref{se:app-connect}, shows that under Condition~\ref{cond-dil} the constructed graph is connected with high probability (w.h.p.). We say that an event $\cE_n$ holds w.h.p. if $\PP(\cE_n)\to 1$ when $n$ tends to infinity.

\begin{lemma}\label{lem:connect}
Consider a random graph $G(n, (d_i)_1^n)$ where the degree sequence $(d_i)_1^n$ satisfies Condition \ref{cond-dil}.
Then $G(n, (d_i)_1^n)$ is connected w.h.p.
\end{lemma}


\subsection{Weighted flooding time and diameter.}\label{sse:main}
We now state our main first result concerning the weighted flooding
time and diameter of $G(n, (d_i)_1^n)$ with i.i.d. exponential $1$
weights on the graph edges.

\begin{theorem}\label{thm-main}
Consider a random graph $G(n, (d_i)_1^n)$ with i.i.d. exponential $1$
weights on its edges, where the degree sequence $(d_i)_1^n$ satisfies
Condition \ref{cond-dil}. Then we have
\begin{center}
\begin{eqnarray*}
\frac{\diam_w(G(n, (d_i)_1^n))}{\log n} &\stackrel{p}{\longrightarrow}& \frac{1}{\nu-1} + \frac{2}{\dmin} \ , \mbox{ \ and \ } \\
\frac{\flood_w(G(n, (d_i)_1^n))}{\log n} &\stackrel{p}{\longrightarrow}& \frac{1}{\nu-1} + \frac{1}{\dmin},
\end{eqnarray*}
\end{center}
where $\stackrel{p}{\longrightarrow}$ is the convergence in probability.
\end{theorem}
In the particular case where $G$ is a random
$r$-regular graph with $r\geq 3$, we recover a result first proved in
\cite{ding09} concerning the weighted diameter.
We refer to \cite{amle} for an extension of this result when $\dmin$
can be smaller than 3.

In order to further compare our result with existing ones, we reproduce here a
result of Bhamidi, van der Hofstad, Hooghiemstra~\cite{BHH09}:
\begin{theorem}[\cite{BHH09}]\label{Th-FPP}
Consider a random graph $G(n, (d_i)_1^n)$ where the degrees $d_i$ satisfy Condition \ref{cond-dil}. Let $a, b$ be two uniformly chosen vertices in this graph. Then there exists a random variable $V$ such that
\begin{eqnarray}\label{eq-fpp}
\dist_w(a,b) - \frac{\log n}{\nu -1} \stackrel{d}{\longrightarrow} V,
\end{eqnarray}
where $\stackrel{d}{\longrightarrow}$ is the convergence in distribution.
\end{theorem}

For completeness, we also include results of {\small van der Hoftstad, Hooghiemstra, Mieghem~\cite{HHM05}} and {\small Fernholz, Ramachandran~\cite{fern07}} for the typical distance and the diameter in $G(n, (d_i)_1^n)$:

\begin{theorem}[\cite{HHM05, fern07}]\label{th:dist}
Consider a random graph $G(n, (d_i)_1^n)$ where the degree sequence $(d_i)_1^n$ satisfies Condition \ref{cond-dil}.
Let $a, b$ be two uniformly chosen vertices in this graph. Then we have
\begin{eqnarray*}
\label{eq:typG}\frac{\dist(a,b)}{\log n} \stackrel{p}{\longrightarrow} \frac{1}{\log \nu}, \ \ \text{and} \ \
\label{eq:diamG}\frac{\diam (G(n, (d_i)_1^n))}{\log n} \stackrel{p}{\longrightarrow} \frac{1}{\log \nu}.
\end{eqnarray*}
\end{theorem}

Our Theorem \ref{thm-main} is obviously consistent with the
asymptotic for the typical weighted distance on random graphs given by (\ref{eq-fpp}). However contrary to the case without weight, we see that the asymptotics for the weighted flooding time or diameter are not the same as for the typical weighted distance between two uniformly chosen vertices.
The appearance of the common factor $\frac{\log n}{\nu -1}$ is quite easy to understand at an heuristic level: if one explores the neighborhood of a given vertex consisting of all vertices at (weighted) distance less than $t$, then this exploration process behaves like a continuous time Markov branching process (see \cite{athney} for a precise definition) which is known to grow exponentially fast like $e^{(\nu -1)t}$. In particular at time $\frac{\log n}{2(\nu-1)}$, it reaches the size of the order of $\sqrt{n}$. In particular, if one considers two such exploration processes started from $a$ and $b$, then by that time they should intersect with great probability. This explains why the typical weighted distance is of the order $\frac{\log n}{\nu-1}$. We give a more precise statement of this heuristic in Section~\ref{sse:proof-bp}.
When considering the weighted flooding time, we consider a case where
one exploration process is started from a typical vertex whereas the
other starting point is chosen in order to get a bad scenario in the
sense that the exploration process started from this vertex grows
slowly. Indeed, the bad scenario corresponds to a starting point
having degree $\dmin$ and large weights on all its incident
edges. This event gives the contribution $\frac{\log n}{\dmin}$. We
refer to Section \ref{sse:low} for a formal treatment of this
argument. Of course, to compute the weighted diameter, one has to
consider a case where both starting points correspond to bad scenarios
and then obtains the contribution $2\frac{\log n}{\dmin}$. We see that if one is interested in passing the information between two typical vertices, it can be achieved in time of the order $\frac{\log n}{\nu-1}$ and there is a price (in time) of $\frac{\log n}{\dmin}$ to pay if one wishes to pass the information to everyone from a typical vertex and another price (in time) of $\frac{\log n}{\dmin}$ to pay if one wishes to pass the information to everyone from a vertex in a worst case scenario.
As shown by Theorem \ref{th:dist}, there is no such discrepancy when
there are no weights, i.e. if the vertices are synchronized so that
the process can be run in slotted rounds. In such a case, the
exploration process behaves like a standard Galton-Watson branching
process and bad scenarios (corresponding to slow growth) have very low
probability \cite{vitali} so that they do not contribute in the large
$n$ limit.

\section{ Broadcasting in random regular graphs }\label{se:br}
In this section we elaborate on another aspect of our result. By
comparing our main Theorem \ref{thm-main} with Theorem \ref{th:dist}
in the case of random regular graphs, we see that the weighted
flooding time is actually smaller than the graph
distance flooding time. As for the diameter, this is also valid if $r\geq 6$. With previous discussion, the
heuristic explanation of this phenomenon is clear: the random weights
introduce variance that allows for the branching process
(approximating the exploration process) to grow faster than without
weights. Even if the weights have an average of one, weights with
small values allow the branching process to grow faster than with
constant weight equal to one. Of course the variability of the weights
has also a drawback when one looks at worst case scenario which
correspond to the factors $\frac{\log n}{\dmin}$. However in the case
of random regular graphs the advantages of variance exceeds its
drawback, and the weighted flooding time is smaller than the
graph-distance flooding time. Note that this will not be always true
in the general case, e.g. when $\nu$ is much bigger than $\dmin$.
We now concentrate on one important practical implication of this
phenomenon.

We consider the asynchronous analogue of the standard {\em phone
  call model}  \cite{P87}. In continuous-time, we assume that each
node is endowed with a Poisson process with rate $1$ and that at the
instants of its corresponding Poisson process a node wakes up and
contacts one of its neighbors uniformly at random. We consider the
well-studied push model. In this model, if a node $i$ holds the message it
passes it to its randomly chosen neighbor regardless of its
state. Note that this may yield an unnecessary transmission (if the
receiver already had the message). As in the case of
the standard discrete-time phone call model, we are interested in the
performance of such an information dissemination routine in terms of
the time it takes to inform the whole population.
We denote this time by $\text{ABT}(G)$ for asynchronous broadcast time
with one initial informed node chosen uniformly at random among the
vertices of $G$.

We restrict ourselves to $r$-regular graphs,
i.e. graphs where each node has degree $r\geq 3$, so that $p_r=1$ in
Condition \ref{cond-dil}.
As shown in Section~\ref{sec:app-cor}, the dynamic evolution of informed nodes then corresponds
to the flooding time with i.i.d. weights on edges distributed
according to an exponential distribution with mean $r$. The fact that
the graph is regular, is crucial to get this property and this is the reason why we require $G$ to be a $r$-regular graph.
Hence, our Theorem~\ref{thm-main} allows us to analyze the asynchronous broadcast algorithm for these graphs and
we get the following corollary:
\begin{corollary}\label{th:nttiming}
Let $G\sim \cG(n,r)$ be a random $r$-regular graph with $n$ vertices.
Then
w.h.p.
$$\text{ABT}(G)=2\left(\frac{r-1}{r-2}\right) \log n + o(\log n).$$
%
\end{corollary}

The classical randomized broadcast model was first investigated by Frieze and Grimmett \cite{FG85}.
Given a graph $G = (V, E)$, initially a piece of information is placed on one of the nodes in $V$.
Then in each time step, every informed node sends the information to another node, chosen independently and uniformly at random among its neighbors. The question now is how many time-steps are needed such that all nodes become informed. Note that this model requires nodes to be synchronized.
It was shown by Frieze and Grimmett \cite{FG85} and Pittel \cite{P87} that for the complete graph $K_n$ the number of steps
needed to inform the whole population scales as $\log_2 n + \log n +
o(\log n)$ with high probability.
Fountoulakis et al. \cite{FHP09} proved that in the case of Erd\H{o}s-R\'enyi random graphs $G(n,p_n)$,
if the average degree, $np_n$, is slightly larger than $\log n$, then the broadcast time essentially coincides with
the broadcast time on the complete graph.
For any $r$-regular graphs it has been shown in \cite{ES09} that this
algorithm requires at least $ \left(\frac{1}{\log (2 - 1/r)} -
  \frac{1}{r \log (1 - 1/r)}\right) \log n + o(\log n)$ rounds to
inform all nodes of the graph, w.h.p. (the randomness comes here from
the choice of the neighbor to which the information is pushed).
Fountoulakis and Panagtotou in
\cite{FK10} have  recently  shown that in the case of random regular
graphs, the process completes in
$\left(\frac{1}{\log (2 (1 - 1/r))} - \frac{1}{r \log (1 -
    1/r)}\right) \log n + o(\log n)$ rounds w.h.p.

\begin{figure}[htb]
\begin{center}
\includegraphics[width=0.8\textwidth, height=0.3\textheight]{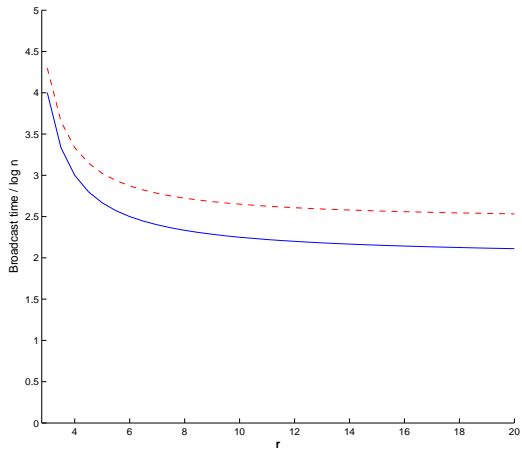}
\end{center}
\caption{Comparison of the time to broadcast in the
  synchronized version (dashed) and with exponential random weights (plain)}
\label{fig:broadcast}
\end{figure}

Note that if instead of independent Poisson clocks of rate one, we
take a deterministic process with slots of size one, our model is
exactly the one studied in \cite{FK10}. Hence locally, both processes
behave similarly: when a node receives the information, it will need
on average a time of
$\left(r+\frac{r}{2}+\dots+\frac{r}{r-1}+1\right)$ to transmit it to
all its neighbors (including possibly informed ones).
Figure \ref{fig:broadcast} shows the comparison between results in
\cite{FK10} and our Theorem \ref{th:nttiming}: in both cases, the time
to broadcast is of the order of $\log n$ but the prefactors differ and
are given by the two curves for various values of $r$.
We see that the asynchronous version is always faster than the synchronized one.
This result while surprising is in agreement with the discussion comparing our Theorem \ref{thm-main} with Theorem \ref{th:dist}. The process of diffusion takes advantage of the variance of the exponential random delays and allows to broadcast the information faster in a decentralized and asynchronous way!

\section{Proof of Theorem~\ref{thm-main}}\label{se:proof}
In this section we present the main steps of the proof of Theorem~\ref{thm-main}.
The main idea of the proof is to grow simultaneously balls from each
vertex so that the diameter is twice the time when the last two balls
intersect. Indeed, instead of taking a graph at random and then analyzing the balls, we use a standard coupling argument in random graphs theory consisting in building the balls and in the same time the graph. We present this coupling in the following Section \ref{sec:explor} and Section~\ref{sec:app-coup}. We then present the branching process approximation in Section \ref{sse:proof-bp}. This section is not technically required for the proof and is not written in a rigorous way. It is included to give some intuition for the proof of the upper bound in Section \ref{sse:up} and the lower bound in Section \ref{sse:low}.

\subsection{The exploration process.}\label{sec:explor}
For $t>0$, we define the $t$-radius neighborhood of a vertex $a$ as
\begin{eqnarray*}
B_w(a,t)&=&\{b,\:\dist_w(a,b)\leq t\},
\end{eqnarray*}
and the first time at which the set $B_w(a,t)$ reaches size $k$ is
denoted by:
\begin{eqnarray*}
T_a(k) = \min \{t: |B_w(a,t)| \geq k \}.
\end{eqnarray*}

Fix a vertex $a$, and consider the continuous-time exploration
process: at time $t=0$, we have a neighborhood consisting only of $a$, and
for $t>0$, the neighborhood is precisely $B_w(a,t)$.

\noindent We now give an equivalent description of the process:
\begin{itemize}
\item Start with $B = \{a\}$, where $a$ has $d_a$ half-edges. Reveal
  any matchings and weights of these $d_a$ half-edges connecting them
  amongst themselves (creating self-loops at $a$). The remaining
  unmatched half-edges are stored in a list $L$.

\item Repeat the following exploration step as long as the list $L$ is
  not empty:
\item[] Given there are $\ell\geq 1$ half-edges in the current list, say
  $L=(h_1,\dots, h_\ell)$, let $\Psi \sim \text{Exp}(\ell)$ be an
  exponential variable with parameter $\ell$. After time $\Psi$ select
  a half-edge from $L$ uniformly at random, say $h_i$. Remove $h_i$
  from $L$ and match it to a
  uniformly chosen half-edge in the entire graph excluding $L$, say $h$. Add
  the new vertex (connected to $h$) to $B$ and reveal the
  matchings (and weights) of any of its half-edges whose match is
  also in $B$. More precisely, let $d$ be the degree of this new
  vertex and $2x$ the number of matched half-edges in $B$ (including
  the matched half-edges $h_i$ and $h$). There is a
  total of $m-2x$ unmatched half-edges. Consider one of the
  $d-1$ half-edges of the new vertex (excluding $h$ which is connected
  to $h_i$); with probability $(\ell-1)/(m-2x-1)$ it is matched with a
  half-edge in $L$ and with the complementary
  probability it is matched with an unmatched half-edge outside
  $L$. In the first case, match it to a uniformly chosen half-edge of
  $L$ and remove the corresponding half-edge from $L$. In the second
  case, add it to $L$. We proceed in the same manner for all the
  $d-1$ half-edges of the new vertex.
\end{itemize}

Let $B(a,t)$ and $L(a,t)$ be respectively the set of vertices and the
list generated by the above procedure at time $t$.
Considering the usual configuration model and using the memoryless
property of the exponential distribution, we have that
$B_w(a,t)=B(a,t)$ for all $t$.

Let $\tau_i$ denote the time of the $i$-th  exploration step in the
above continuous-time exploration process, i.e.
$T_a(i)=\tau_{i-1}$.
Assuming $L(a,\tau_i)$ is not empty, at time $\tau_{i+1}$, we match a uniformly
chosen half-edge from the set $L(a,\tau_i)$ to a uniformly chosen
half-edge among all other half-edges, excluding those in
$L(a,\tau_i)$. Let $\cF_{t}$ be the $\sigma$-field generated by the
above process until time $t$.
Given $\cF_{\tau_i}$, $\tau_{i+1} - \tau_i$ is an
exponential random variable with rate $S_i(a)=|L(a,\tau_i)|$ the size
of the list  consisting of unmatched half-edges in $B(a,\tau_i)$.
Let $i^*=\min\{i,\: S_i(a)=0\}\leq n$, then we set
$S_i(a)=0$ for all $i^*\leq i\leq n$.

For $i\in [1,i^*]$, we define $\hd_i$ the forward degree i.e. the
degree minus one, of the vertex added during the $i$-th exploration step and let
\begin{eqnarray}
\label{eq:defhS}\hS_i(a) := d_a + \hd_1 + ... + \hd_i - i .
\end{eqnarray}
For a connected set $H$, we denote by $\tx(H)$ the tree excess of
$H$ which is the maximum number of edges that can be deleted from the
induced subgraph on $H$ while still keeping it connected.
Let $X_i(a) : = \tx(B(a,\tau_i))$ for $i\leq i^*$, so that we have for
$i\leq i^*$:
\begin{eqnarray}
\label{eq:S(a)}S_i(a) = \hS_i(a)  - 2 X_i(a).
\end{eqnarray}
Note that $i^*$ and the sequence $\hd_i$
depend on $a$.

We define
\begin{eqnarray}
\label{eq:defab}\alpha_n  = \log^3 n \mbox{ and, } \beta_n=
3\sqrt{\frac{\lambda}{\nu-1} n \log n}.
\end{eqnarray}

An important ingredient in the proof will be the comparison of the variables $\{\hd_1 , ... , \hd_k \}$, for an appropriately chosen $k$ , to an i.i.d. sequence.

\subsection{Coupling the forward degrees sequence $\hd_i$.}\label{sec:app-coup}

The sequence $(\hd_i)_{i\leq i^*}$ can be constructed as follows.
Initially, associate to vertex $j$ a bin containing a set of $d_j$
white balls. At step $0$, color the balls corresponding to vertex $a$
in red. Subsequently, at step $k \leq i^*$, choose a ball uniformly at
random among all white balls; if the ball is drawn from node $j$'s bin
then set $\hd_k = d_{j} - 1 $, and color all the balls in the
bin in red. If $i^*<n$, there are still white balls at step $i^*+1$ and
we complete the sequence $\hd_i$ for $i\in [i^*+1,n]$ by continuing
the sampling described above, so that we obtain a sequence
$(\hd_i)_{i=1}^n$ coinciding with the sequence defined in Section
\ref{sec:explor} for $i\leq i^*$. We also extend the sequence $\hS_i(a)$ for $i> i^*$
thanks to (\ref{eq:defhS}) and we set $X_i(a)=X_{i^*}(a)$ for all
$i>i^*$. Note that with these conventions, the relation
(\ref{eq:S(a)}) is not valid for $i>i^*$ but we have $S_i(a)\leq
\hS_i(a)-2X_i(a)$.

We now present a coupling of the variables $\{\hd_1 , ... , \hd_k \}$
valid for $k\leq \beta_n$ defined in (\ref{eq:defab}), with an
i.i.d. sequence, that we now define.
First, we denote the order statistics of the degrees of the nodes of
the graph by
\begin{eqnarray*}
d_{(1)}\leq d_{(2)}\leq \dots \leq d_{(n)}\:.
\end{eqnarray*}
Let $\um^{(n)}=\sum_{i=1}^{n-\beta_n} d^{(n)}_{(i)}$ and $\upi^{(n)}$
the size-biased empirical distribution with the $\beta_n$ highest
degrees removed, i.e.
\begin{eqnarray*}
\upi^{(n)}_k = \frac{\sum_{i=1}^{n-\beta_n} (k+1) \ind_{\left(d^{(n)}_{(i)}=k+1\right)}}{\um^{(n)}}.
\end{eqnarray*}
Similarly, let $\bm^{(n)}=\sum_{i=(\beta_n+1)\Delta_n}^{n}
d^{(n)}_{(i)}$ and $\bpi^{(n)}$ the size-biased empirical distribution
with the $(\beta_n+1)\Delta_n$ lowest degrees removed, i.e.
\begin{eqnarray*}
\bpi^{(n)}_k = \frac{\sum_{i=(\beta_n+1)\Delta_n}^{n} (k+1) \ind_{\left(d^{(n)}_{(i)}=k+1\right)}}{\bm^{(n)}}.
\end{eqnarray*}
Note that by Condition \ref{cond-dil}, we have
$\beta_n\Delta_n=o(n)$ implying that both distributions $\upi^{(n)}$
and $\bpi^{(n)}$ converge to size biased distribution $q$ defined in
(\ref{eq:defq}).

For two real-valued random variables $A$ and $B$, we write $A\leq_{st} B$ if
for all $x$, we have $\PP(A>x)\leq \PP(B>x)$. If $C$ is another random
variable, we write $A\leq_{st} B|C$ if for all $x$, $\PP(A>x)\leq
\PP(B>x|C)$ a.s.


\begin{lemma}\label{lem-coupl}
For an uniformly chosen vertex $a$ and for $i\leq \beta_n$, we have
\begin{eqnarray}
\label{eq:coupl}\uD^{(n)}_i \leq_{st} \hd_i | (d_a,\hd_1,\dots, \hd_{i-1})\leq_{st}\bD^{(n)}_i
\end{eqnarray}
where $\uD^{(n)}_i$ (resp. $\bD^{(n)}_i$) are i.i.d. with distribution
$\upi^{(n)}$ (resp. $\bpi^{(n)}$).
In particular, we have
\begin{eqnarray*}
\sum_{k=1}^{i} \uD^{(n)}_k \leq_{st} \sum_{k=1}^{i} \hd_k \leq_{st} \sum_{k=1}^{i} \bD^{(n)}_k.
\end{eqnarray*}
\end{lemma}
\begin{proof}
We fix the sequence of degrees $d_{i}^{(n)}$ and the initial vertex
$a$. We now prove that conditionally on the values
of $(d_a,\hd_1,...,\hd_{j-1})$, the random variable $\hd_j$ is
stochastically smaller than $\bD^{(n)}_j$.
This can be seen by a simple coupling argument as follows.
First order the balls from $1$ to $m$ consistently with the order
statistics, i.e. start by numbering the balls in the bin with the
fewest balls and then move to the larger ones as ordered by the number
of balls they contain.

Given the sequence $(d_a,\hd_1,...,\hd_{j-1})$, color in red the balls
of bins of the corresponding sizes. In order to get a sample for $\bD^{(n)}_j$,
pick a ball at random among all balls in the last $n-(\beta_n+1)\Delta_n$ bins and
set $\bD^{(n)}_j$ to be equal to the size of the selected bin minus one.
If the ball picked is white set $\tilde{d}_j=\bD^{(n)}_j$.
If there are red balls in the last $n-(\beta_n+1)\Delta_n$ bins and if
such a ball is picked, say this is the $\ell$-th ball among these red
balls for the induced order, then set $\tilde{d}_j$ to be the size of
the bin containing the $\ell$-th white ball, minus one. Since $d_a+\hd_1+\dots+\hd_{j-1}\leq \beta_n\Delta_n$,
this ball is in one of the first $(\beta_n+1)\Delta_n$ bins.
In all cases, we have: $\hd_j\leq_{st}\tilde{d}_j\leq \bD^{(n)}_j$
given the sequence $(d_a,\hd_1,...,\hd_{j-1})$.
A similar argument allows to prove that $\uD^{(n)}_j \leq_{st} \hd_j$
given the sequence $(d_a,\hd_1,...,\hd_{j-1})$.

The second statement follows from the following lemma \cite[Lemma A.3]{fern07} .

\begin{lemma}\label{lem-domin}
Let $X_1,...,X_t$ be a random process adapted to a filtration $\mathcal{F}_0 = \sigma[\emptyset], \mathcal{F}_1,...,\mathcal{F}_t$, and let ${\bf \Sigma}_t = X_1+ ...+X_t$. Consider a distribution $\mu$ such that $({X}_{s+1}| \mathcal{F}_s) \geq_{st} \mu$ (resp. $({X}_{s+1}| \mathcal{F}_s) \leq_{st} \mu$) for all $0 \leq s \leq t-1$. Then ${\bf \Sigma}_t$ is stochastically greater (resp. smaller) than  the sum of $t$ i.i.d. $\mu$-distributed random variables.
\end{lemma}
\end{proof}
\subsection{Branching process approximation.}\label{sse:proof-bp}

In this section, we consider the continuous-time Markovian branching process $Z_t$ approximating the exploration process defined in Section~\ref{sec:explor}.
As we will see, the tree excess $X_i(a)$ remains small (compared to $\hS_i(a)$) with very high probability at least when $i$ is not too large. The branching process approximation consists in neglecting this term and considering that the sequence of $\hd$ is a sequence of i.i.d. random variables with distribution given by (\ref{eq:defq}) (which is true asymptotically).
$Z_t$ is started
with one ancestor. Each of the members of the population has an exponential lifetime
(with mean one) and upon her death she gives birth to a random number $\hD$ of
particles, where $\hD$ has distribution (\ref{eq:defq}).
We assume that $\dmin\geq 3$ so that $\hD\geq 2$ a.s.
We now define the split times: the times at which the particles split
(see \cite{athney} III.9).
Let $\hSigma_i=D+\sum_{j=1}^i \hD_j-i$ (where $D$ is the degree of the vertex the process starts from, i.e. $\PP(D=r) = p_r$) and $E_i$ a sequence of independent
exponential random variables with mean one. Under the assumption $\dmin\geq 3$, the
split times are defined by $T_0=0$ and for $i\geq 1$,
\begin{eqnarray}\label{eq:Ti}
T_i = \sum_{j=0}^{i-1}\frac{E_i}{\hSigma_i}.
\end{eqnarray}
Note that $\EE[\hSigma_i]\approx (\nu-1)i$ and it is shown in
\cite{athney} that:
$\lim_{n\to\infty}\frac{T_n}{\log n} = \frac{1}{\nu-1}$.
In particular at time $\frac{\log n}{2(\nu-1)}$, the process reaches
size $\sqrt{n}$, so that for two given vertices of the graph, there is a
high probability that the two balls intersect (see Proposition \ref{prop-up} below). This heuristic argument allows
to understand the typical distance given by (\ref{eq-fpp}) in \cite{BHH09}.

In order to be able to compute the diameter, we need to find $x$ such
that $\PP(T_{\sqrt{n+}}\geq x)\approx \frac{1}{n}$ (here $n+$ is an
informal notation to denote a sequence growing slightly faster than
$n$, like $n\log n$, see the parameter $\beta_n$). Hence we need to study large deviations results for the sequence $\{T_i\}_{i\in \NN}$. Moreover we have to take care of the error introduced by the branching process approximation. This is done by a coupling argument given in Section~\ref{sec:app-coup}. We now give the main technical steps of the proof.

\subsection{Proof of the upper bound.}\label{sse:up}
The following proposition, proved in Section~\ref{sec:app-up}, shows that to bound the
distance between any pair of nodes it suffices to bound the time when
$\beta_n := 3\sqrt{\frac{\lambda}{\nu-1}n\log n}$ nodes have been reached in the exploration process defined in Section \ref{sec:explor} starting from these two nodes.
\begin{proposition}\label{prop-up}
We have w.h.p.
\begin{eqnarray*} \dist_w(u,v) \leq T_u(\beta_n) + T_{v}(\beta_n) \mbox{ , for all
    $u$ and $v$}.\end{eqnarray*}
\end{proposition}

\noindent Now we give an upper bound (which holds w.h.p.) for the time needed to explore $\beta_n$ nodes from any vertex $a$.

\begin{proposition}\label{prop-up-bis}
For a uniformly chosen vertex $u$ and any $\epsilon>0$, we have
\begin{eqnarray*}
\PP\left( T_u(\beta_n) \geq \frac{(1+\epsilon)\log n}{2(\nu-1)} + l \right) = o(n^{-1} + e^{-\dmin l})  .
\end{eqnarray*}
\end{proposition}
\noindent A proof of this proposition is given in Section~\ref{sec:app-up-bis}. We give here a heuristic based on the branching process approximation defined in the last section.
Note that for $T_i$ defined by (\ref{eq:Ti}), we have
\begin{eqnarray*}
\EE\left[ e^{\theta T_n}| \hSigma\right] = \prod_{i=0}^{n-1}\left(
  1+\frac{\theta}{\hSigma_i-\theta}\right)\leq \exp\left( \sum_{i=0}^{n-1}
\frac{\theta}{\hSigma_i-\theta}\right).
\end{eqnarray*}
Then for small values of $i\leq \alpha_n$, where $\alpha_n=\log^3 n$ grows slower than any
power of $n$, we use the lower bound:
\begin{eqnarray*}
\hSigma_i\geq i(\dmin-2)+\dmin,
\end{eqnarray*}
so that we get with $\theta=\dmin$:
\begin{eqnarray*}
\PP(T_{\alpha_n}\geq x\log n) &\leq& \EE\left[ e^{\dmin T_{\alpha_n}}\right]\exp(-x\dmin
\log n) \\ &\approx& \alpha_n^{\frac{\dmin}{\dmin-2}}n^{-x\dmin}.
\end{eqnarray*}
Now for larger values of $i$, we can use
the approximation $\hSigma_i\geq i(\nu-1-\epsilon)$ so that we get
\begin{eqnarray*}
\PP(T_{\beta_n} -T_{\alpha_n} \geq y\log n) \leq \EE\left[ e^{\theta
    (T_{\beta_n}-T_{\alpha_n})}\right] \exp (-\theta y\log n)\\
\approx \left(
  \frac{\beta_n}{\alpha_n}\right)^{\frac{\theta}{\nu-1-\epsilon}}
n^{-y\theta}\approx n^{\theta\left( \frac{1}{2(\nu-1-\epsilon)}-y\right)},
\end{eqnarray*}
so if we choose $y\approx \frac{1}{2(\nu-1)} $, we get  \begin{eqnarray*}
\PP\left(T_{\beta_n} - T_{\alpha_n} \geq
\left(\frac{1}{2(\nu-1)}\right)\log n\right)\approx n^{-1}.
\end{eqnarray*}
We refer to Section~\ref{sec:app-up-bis} for the technical details.

\noindent By Proposition~\ref{prop-up-bis}, we have for
a uniformly chosen vertex $u$ and any $\epsilon>0$:
\begin{eqnarray*}
\PP\left(T_u(\beta_n) \geq
  \left(\frac{1+\epsilon}{2(\nu-1)}\right)\log
  n\right) = o(1), \mbox{ and } \\
\PP\left(T_u(\beta_n) \geq
  \left(\frac{1+\epsilon}{2(\nu-1)}+\frac{1+\epsilon}{\dmin}\right)\log
  n\right) = o(n^{-1}).
\end{eqnarray*}
By taking a union bound over $u$, it follows that
\begin{eqnarray*}
\PP\left(T_u(\beta_n) \leq
  \left(\frac{1+\epsilon}{2(\nu-1)}+\frac{1+\epsilon}{\dmin}\right)\log
  n,\mbox{
      for all $u$}\right) \\ = 1-o(1).
\end{eqnarray*}
Combining this with Proposition \ref{prop-up} finishes the proof of the upper
bound.

\subsection{Proof of the lower bound.}\label{sse:low}
To prove the lower bound, it suffices to prove that for any $\epsilon > 0$,  there exists w.h.p. two vertices $u$ and $v$ such that
$$\dist_w(u,v) \geq \left(\frac{1}{\nu-1} + \frac{2}{\dmin}\right)(1-\epsilon) \log n .$$
Let $G_n \sim G(n, (d_i)_1^n)$, and
$$B'_w(u,t) := \left\{ v: \dist_w(N(u),v) \leq t \right\},$$ where $N(u)$ denotes the neighbors of $u$ in $G_{n} $ . Let $V_{\dmin}$ be the set of vertices with degree $\dmin$. The following proposition, proved in Section~\ref{sec:app-low}, gives a lower bound for the distance between the neighbors of a uniformly chosen vertex $u \in V_{\dmin}$ from the neighbors of another uniformly chosen vertex $v \in V_{\dmin}$.
\begin{proposition}\label{prop-low}
Let $u$, $v$ be two uniformly chosen vertices of the graph $G_n$, with degree $\dmin$, i.e. $u,v \in V_{\dmin}$, and let $t_n = \frac{(1-\epsilon)\log n}{2(\nu-1)} $. We have  w.h.p.
\begin{eqnarray*}
B'_w(u,t_n)\cap B'_w(v,t_n) = \emptyset .
\end{eqnarray*}
\end{proposition}
Now let $s_n:= \frac{1-\epsilon}{\dmin} \log n $, and call a vertex in $V_{\dmin}$ {\it bad} if the weights on all the $\dmin$ edges connected to it are larger than $s_n$. Let $A_u$ denote the event that $u$ is bad, and let $Y:=\sum_u \ind_{A_u}$ be the number of bad vertices. Then we have for $u \in V_{\dmin}$:
$$\PP(A_u) = \PP(\text{Exp}(d_{min}) \geq s_n) = n^{-1+\epsilon}. $$
Then it is easy to see that
\begin{eqnarray*}
\EE(Y) &=& \sum_u \PP (A_u) = p_{\dmin} (1+o(1)) n^{\epsilon}, \ \mbox{and} \ \\
\text{Var}(Y) &=& \sum_{u,v \in V_{\dmin}} \text{Cov}\left(\ind_{A_u},\ind_{A_v}\right) \\
&=& \sum_{u\in V_{\dmin}} \text{Var}\left(\ind_{A_u}\right) + \sum_{uv\in E} \text{Cov}\left(\ind_{A_u},\ind_{A_v}\right)\\
&\leq& (\dmin+1) \EE(Y) .
\end{eqnarray*}
Then by Chebyshev's inequality, we have $$Y \geq \frac{2}{3} p_{\dmin} n^{\epsilon}$$ with high probability.

Let $Y'$ denote the number of bad vertices that are of distance at most $s_n+\frac{(1-\epsilon)}{\nu -1} \log n$ from vertex $a$ (chosen uniformly). By Proposition \ref{prop-low} for a uniformly chosen vertex $i$ we have w.h.p. $B'_w(a,t_n) \cap B'_w(i,t_n) = \emptyset$. In particular, conditioning on the event $A_i$, the probability that $B'_w(i,t_n)$ does not intersect $B'_w(a,t_n)$ remains the same. Hence, for a uniformly chosen vertex $i$ we have
$$\PP\left(A_i, B'_w(a,t_n) \cap B'_w(i,t_n) \neq \emptyset\right) = o(\PP(A_i)), $$ and then we deduce
$$\EE(Y') = o(\EE(Y)) = o(\log n) .$$
By Markov's inequality, $Y' \leq \frac{1}{3} p_{\dmin} n^{\epsilon}$ w.h.p., and hence $Y-Y'$ is w.h.p. positive. This implies the existence of a vertex $i$ whose distance from $a$ is at least $\left( \frac{1}{\nu-1} + \frac{1}{\dmin}\right)(1-\epsilon) \log n $. Then for any $\epsilon > 0$ we have w.h.p.
\begin{eqnarray*}
\flood_w(G_n) &:=& \max_{1 \leq i \leq n} \dist_w(a,i) \\
&\geq& \left( \frac{1}{\nu-1} + \frac{1}{\dmin}\right)(1-\epsilon) \log n .
\end{eqnarray*}

Let $R$ denote the number of pairs of distinct bad vertices. Then $Y \geq \frac{2}{3} p_{\dmin} n^{\epsilon}$ gives
$$R \geq \frac{1}{4} p_{\dmin}^2 n^{2\epsilon}$$ w.h.p.
By Proposition \ref{prop-low} for two uniformly chosen vertices $u, v$ we have w.h.p. $$B'_w(u,t_n) \cap B'_w(v,t_n) = \emptyset.$$ In particular, conditioning on the events $A_u$, and $A_v$, the probability that $B'_w(u,t_n)$ does not intersect $B'_w(v,t_n)$ remains the same. Hence, for two uniformly chosen vertices $u, v$ we have,
$$\PP\left(A_u, A_v, B'_w(u,t_n) \cap B'_w(v,t_n) \neq \emptyset\right) = o(\PP(A_u,A_v)) .$$

Let $R'$ denote the number of pairs of bad vertices that are of distance at most $2s_n+\frac{(1-\epsilon)}{\nu-1} \log n$. Then we have
$$\EE R' = o(\EE Y^2) = o (n^{2\epsilon}).$$ By Markov's inequality, $R' \leq \frac{1}{6} p_{\dmin}^2 n^{2\epsilon} $ w.h.p, and hence $R-R'$ is w.h.p positive. This implies that for any $\epsilon > 0$ we have w.h.p.
\begin{eqnarray*}
\diam_w(G_n) &:=& \max_{1 \leq i,j \leq n} \dist_w(i,j) \\ &\geq& \left( \frac{1}{\nu-1} + \frac{2}{\dmin} \right)(1-\epsilon) \log n ,
\end{eqnarray*}
which completes the proof.

\subsection{Proof of Corollary~\ref{th:nttiming}.}\label{sec:app-cor}

First we prove that for a $r$-regular graph $G=(V,E)$ the dynamic evolution of
informed nodes in continuous-time broadcast when each node is endowed with
a Poisson process with rate 1 corresponds exactly to the flooding time
with exponential random weights on edges with mean $r$. Let $\mathcal{I}(a,t)$
denote the set of informed nodes at time $t$ when we start the broadcast
process from node $a \in V$. Indeed we show that random map $\mathcal{I}(a,.)$
from $[0,\infty)$ to subsets of $V$ has the same law as $B_w(a,.)$ when
the weights are exponential with mean $r$ using a coupling argument:
from the asynchronous broadcasting model, we construct weights on the
edges of the graph and show that these weights are independent
exponential with mean $r$.


 Let $T(v)$ denote the time at which node $v$ becomes informed in the
asynchronous broadcast model and let $\tau_i(u,v)$, for $(u,v)\in E$,
denote the $i$-th time that node $u$ contacts node $v$.
Now we define the weight of the edge $e=(u,v)$ as follows:
\begin{itemize}
\item if $T(u)\leq T(v)$ then
$$w_e := \min_i \{ \ \tau_i(u,v) - T(u) \ | \ \tau_i(u,v) > T(u) \ \} .$$
\item if $T(u)> T(v)$ then
$$w_e := \min_i \{ \ \tau_i(v,u) - T(v) \ | \ \tau_i(v,u) > T(v) \ \} .$$
\end{itemize}

Thanks to the memoryless property of the Poisson process, $\{w_e, e \in E\}$
are independent exponential random variables with mean $r$ and are
such that we have $\mathcal{I}(a,t)=B_w(a,t)$ for all $t\geq 0$.

Hence the asynchronous broadcast time corresponds to the flooding time
with exponential weights with mean $r$ and it
is easy to conclude the proof by Theorem~\ref{thm-main}, that is w.h.p., 
\begin{eqnarray*}
\text{ABT}(G) &=& r \left( \frac{1}{r-2} + \frac{1}{r} \right) \log n + o(\log n) \\
&=& 2\left(\frac{r-1}{r-2}\right) \log n + o(\log n).
\end{eqnarray*}

\bibliographystyle{abbrv}
\bibliography{Biblio}

\iftrue

\appendix
\section{Proofs}\label{se:appendix}

Recall that $\beta_n=3\sqrt{\frac{\lambda}{\nu-1} n \log n}$, and $\alpha_n=\log^3 n$. Now we consider the exploration process defined in Section~\ref{sec:explor}.

\subsection{Proof of Proposition~\ref{prop-up-bis}.}\label{sec:app-up-bis}

Before getting to the main part of the proof, we need to prove some
technical lemmas.

We start by some simple remarks.
The process $i\mapsto X_i(a)$ is non-decreasing in
$i\in [1,n]$. Moreover, given $\cF_i$, the increment $X_{i+1}(a)-X_i(a)$ is
stochastically dominated by a binomial variable:
\begin{eqnarray}
\label{upinc1}\Bin\left( \hd_{i+1}, \frac{(S_i(a)-1)^+}{m^{(n)}-2(X_i(a)+i)}\right).
\end{eqnarray}
Note that if $i>i^*$, then $S_i(a)=0$ and $X_{i+1}(a)-X_i(a)=0$ so
that (\ref{upinc1}) is still valid.
For $i< n/2$, we have
\begin{eqnarray*}
\frac{(S_i(a)-1)^+}{m^{(n)}-2(X_i(a)+i)}&\leq& \frac{\hS_i(a)-2X_i(a)}{m^{(n)}-2(X_i(a)+i)}\\
&\leq&\frac{\hS_i(a)}{m^{(n)}-2i}\\
&\leq&\frac{\max_{\ell\leq i}\hS_\ell(a)}{n-2i}.
\end{eqnarray*}
Hence, we obtain for $i< n/2$:
\begin{eqnarray}
\label{upinc11} && X_i(a)\leq_{st}\Bin\left(\max_{\ell\leq i}\hS_\ell(a)+i, \frac{\max_{\ell\leq i}\hS_\ell(a)}{n-2i}\right).
\end{eqnarray}

\begin{lemma}\label{lem:upX}
For any $k<n/2$, we have
\begin{eqnarray*}
\PP\left(2X_k(a)\geq x\mid \hS_k(a), i^*\geq k\right) \leq   \\ \PP\left(
  \Bin\left(\hS_k(a),\sqrt{\hS_k(a)/n}\right)\geq x \mid \hS_k(a)\right).
\end{eqnarray*}
\end{lemma}
\begin{proof}
Note that given a set of $k$ edges connecting the $k+1$ vertices in $B_w(a, T_a(k+1))$ and given
$\hS_k(a)$, the remaining edges of the graph are obtained by an
uniform matching of the remaining half-edges; the total remaining
number of half-edges is $m^{(n)}-2k$ (which is greater than $n$ by $\dmin \geq 3$) and the number of unmatched half-edges
in $B_w(a, T_a(k+1))$ is exatly $\hS_k(a)$.

The proof follows from the following lemma proved in \cite[Lemma 3.2]{fern07}:
\begin{lemma}\label{lem-excess}
Let $A$ be a set of $m$ points, i.e. $|A|=m$, and let $F$ be a uniform random matching of elements of $A$. For $a \in A$, we denote by $F(a)$ the point matched to $a$, and similarly for $B \subset A$, we write $F(B)$ for the set of points matched to $B$. Now assume $B \subset A$, and let $k=|B|$. We have
\begin{eqnarray*}
|B \cap F(B)| \leq_{st} \Bin (k,\sqrt{k/m}).
\end{eqnarray*}
\end{lemma}
\end{proof}

\noindent We define the events :
\begin{eqnarray*}
R(a) := \left\{ S_k(a) \geq \dmin + (\dmin-2)k ,\mbox{ for all } 1
  \leq k \leq \alpha_n \right\} , \\
R'(a) := \left\{ S_k(a) \geq 1 + (\dmin-2)k ,\mbox{ for all } 1 \leq k \leq \alpha_n \right\} .
\end{eqnarray*}

\begin{lemma}\label{lem-R1}
For a uniformly chosen vertex $a$, we have
\begin{eqnarray}
\PP(R(a)) &\geq& 1-o((\log^{10} n / n)) , \label{pr-R1}\\
\PP(R'(a)) &\geq& 1 - o(n^{-3/2}) . \label{pr-R2}
\end{eqnarray}
\end{lemma}

\begin{proof}
Note that Since $\dmin\geq 3$, the sequence $\hS_k(a)$ is
non-decreasing in $k$.
We also have for all $k\leq \alpha_n$,
$$\dmin + (\dmin-2)k \leq \hS_k(a) \leq \alpha_n\Delta_n = o(n) .$$
Hence we have
\begin{eqnarray*}
\{X_k(a)=0\}\subset R(a), \quad \{X_k(a)\leq 1\}\subset R'(a).
\end{eqnarray*}
We distinguish two cases:
\begin{itemize}
\item {\bf Case 1:} $\hS_{\alpha_n}(a) < 2 \dmin  \alpha_n$.
Then given this event denoted by $Q_1$, by (\ref{upinc11}), we have
\begin{eqnarray*}
X_{\alpha_n}(a) \leq_{st} \Bin\left( (2\dmin +1)\alpha_n, \frac{2\dmin
    \alpha_n}{n-2\alpha_n}\right).
\end{eqnarray*}
Hence we have
\begin{align*}
\PP\left(X_{\alpha_n}(a) \geq 1 \mid Q_1 \right) \leq
O(\alpha_n^2/n),
\end{align*}
and
\begin{align*}
\PP\left(X_{\alpha_n}(a) \geq 2 \mid Q_1 \right) \leq
O(\alpha_n^4/n^2).
\end{align*}
Then we have
\begin{eqnarray*}
\PP(R(a)^c \mid Q_1) &\leq & O(\alpha_n^2/n),\\
\PP(R'(a)^c \mid Q_1) &\leq & O(\alpha_n^4/n^2).
\end{eqnarray*}

\item {\bf Case 2.} if $\hS_{\alpha_n}(a) \geq 2 \dmin  \alpha_n$, we
  have $$\hS_{\alpha_n}(a)\leq \alpha_n\Delta_n = o(n).$$
Let $k=\max\{i,\:\hS_i(a)< 2\dmin \alpha_n\}$. Then on the event
$Q_1^c$, we have $k\leq \alpha_n$ and by a similar argument as in case
1, given the event $Q_1^c$, we obtain
\begin{eqnarray*}
X_k(a)\leq_{st} \Bin\left(
  (2\dmin+1)\alpha_n,\frac{2\dmin\alpha_n}{n-2\alpha_n}\right).
\end{eqnarray*}
Then given the event $Q_1^c$, we have by (\ref{upinc11}):
$$X_{\alpha_n}(a) \leq_{st} \Bin\left(\alpha_n(\Delta_n+1) ,
  \frac{\alpha_n\Delta_n}{n-2\alpha_n} \right).$$

Then letting $M = \lceil 2 \tau^{-1}\rceil$, and by Chernoff's inequlaity we have
\begin{eqnarray*}
\PP\left(X_{\alpha_n}(a) \geq M \mid Q_1^c \right)  \leq
 O\left( (\Delta^2_n \alpha^2_n / n)^M \right) =  o(n^{-3}) .
\end{eqnarray*}
Note that for $n$ large enough $$2\dmin \alpha_n - 2M > \dmin +
(\dmin-2)\alpha_n.$$
Hence we have
\begin{eqnarray*}
\{X_k(a)=0, \:X_{\alpha_n}(a)\leq M,\:Q_1^c\} &\subset& R(a)\cap
Q_1^c\mbox{,}  \\
\{X_k(a)\leq 1, \:X_{\alpha_n}(a)\leq M,\:Q_1^c\}
&\subset& R'(a)\cap Q_1^c
\end{eqnarray*}
Then we have
\begin{eqnarray*}
\PP\left(R(a)^c\mid Q_1^c \right) \leq
O(\alpha_n^2/n), \\
\PP\left(R'(a)^c\mid Q_1^c \right) \leq
O(\alpha_n^4/n^2).
\end{eqnarray*}
\end{itemize}
The lemma follows.
\end{proof}

We will use the following properties: if $Y$ is an exponential random
variable with mean $\mu^{-1}$, then for any $\theta<\mu$, we have
$\EE\left[ e^{\theta Y}\right] = \frac{\mu}{\mu-\theta}$.
Given the sequence $S_k(a)$, for $k<i^*$, the random variables
$\tau_{k+1}-\tau_k$ are iid exponential random variables with mean
$S_k(a)^{-1}$ .

\begin{lemma}\label{lem-up-r}
For a uniformly chosen vertex $a$, and any $\epsilon>0$, we have
\begin{eqnarray*}
\PP\left( T_a(\alpha_n)\geq \epsilon \log n + l \right) = o(n^{-1} + e^{-\dmin l}) .
\end{eqnarray*}
\end{lemma}

\begin{proof}
Assume that $R'(a)$ holds and consider the following two cases:

\noindent {\bf Case 1: $R(a)$ holds.}

\noindent In this case, for any $k<\alpha_n$ we have conditioning on $R(a)\cap R'(a)$,
$$\tau_{k+1} - \tau_k \leq_{st} Y_k = \text{Exp}\left(\dmin+(\dmin-2)k\right),$$
and all $Y_k$'s are independent.
Hence, we have

\begin{eqnarray*}
\EE\left[ e^{\dmin (T_a(\alpha_n) - T_a(1))}\mid R(a)\cap R'(a)\right] \\
\leq \prod_{i=1}^{\alpha_n-1}\left(1 + \frac{\dmin}{(\dmin-2)i}\right)  \\
\leq \exp\left[\frac{\dmin}{\dmin-2} \sum_{i=1}^{\alpha_n-1} \frac{1}{i} \right] \\
\leq \alpha_n^{\dmin} = (\log n)^{3\dmin} ,
\end{eqnarray*}
for $n$ large enough. Then for $\epsilon > 0$, by Markov's inequality, we have
\begin{eqnarray*}
\PP\left( T_a(\alpha_n) - T_a(1) \geq \epsilon \log n + l \mid
  R(a)\cap R'(a)\right) \\ \leq (\log
n)^{3\dmin}\exp(-\dmin(\epsilon\log n +l) )\\
= \frac{(\log n)^{3\dmin}}{n^{\epsilon \dmin}}e^{-\dmin l}
= o(e^{-\dmin l}).
\end{eqnarray*}
We also have $T_a(1) \leq_{\text{st}} \text{Exp} (\dmin)$; hence
$$\PP(T_a(1)\geq \epsilon \log n + l)\leq
\frac{\exp(-l\dmin)}{n^{\epsilon \dmin}},$$ and we have
\begin{eqnarray*}
\PP\left( T_a(\alpha_n)\geq \epsilon \log n + l \mid R(a)\cap R'(a)
\right)= o(e^{-\dmin l}).
\end{eqnarray*}

\noindent {\bf Case 2: $R(a)$ doesn't hold.}

\noindent In this case, for any $k<\alpha_n$ we have conditioning on $R(a)^c\cap R'(a)$,
$$\tau_{k+1} - \tau_k \leq_{st} Y_k = \text{Exp}(1 + (\dmin-2)k),$$
and all the $Y_k$'s are independent.
We have
\begin{eqnarray*}
\EE \left[ e^{(T_a(\alpha_n) - T_a(1))}\mid R(a)^c\cap R'(a) \right] \\ \leq \prod_{i=1}^{\alpha_n-1}\left(1 + \frac{1}{(\dmin-2)i}\right)\\
\leq \exp\left[\frac{1}{\dmin-2} \sum_{i=2}^{\alpha_n-1} \frac{1}{i} \right] \\
\leq \alpha_n = \log^3 n,
\end{eqnarray*}
for $n$ large enough. Again by Markov's inequality, we have
\begin{eqnarray*}
\PP\left( T_a(\alpha_n) - T_a(1) \geq \epsilon \log n + l \mid R(a)^c\cap R'(a)\right) \\ \leq \log^3 n \exp (- \epsilon \log n - l ) =  o(n^{-\epsilon/4}).
\end{eqnarray*}
We also have $T_a(1) \leq_{\text{st}} \text{Exp} (1)$, and we conclude in this case
\begin{eqnarray*}
\PP\left( T_a(\alpha_n) \geq \epsilon \log n + l \mid R(a)^c\cap R'(a)
\right) = o(n^{-\epsilon/4}) .
\end{eqnarray*}

\noindent Putting all these together we have
\begin{eqnarray*}
\PP\left( T_a(\alpha_n)\geq \epsilon \log n + l \right) \leq 1 - \PP(R'(a)) +\\  (1 - \PP(R(a))) n^{-\epsilon/4} + o(e^{-\dmin l} )  \\ \leq  o(n^{-1}+ e^{-\dmin l}),
\end{eqnarray*}
as desired.
\end{proof}

We continue with a simple large deviation estimate.
\begin{lemma}\label{lem-upp}
Let $\uD^{(n)}_i$ be i.i.d. with distribution $\upi^{(n)}$. For any $\eta < \nu$,
there is a constant $\gamma > 0$  such that for $n$ large enough we have
\begin{eqnarray}\label{eq-upp}
\PP\left(\uD^{(n)}_1+\dots+\uD^{(n)}_k\leq k \eta \right)\leq e^{-\gamma k}.
\end{eqnarray}
\end{lemma}
\begin{proof}
Let $D^*$ be a random variable with distribution $\PP(D^* = k) =
q_k$ given in (\ref{eq:defq}) so that $\EE[D^*]=\nu$.
Let $\phi(\theta) = \EE[e^{-\theta D^*}]$. For any
$\epsilon>0$, there exists $\theta_0>0$ such that for any $\theta\in
(0,\theta_0)$, we have
\begin{eqnarray*}
\log \phi(\theta)< (- \nu + \epsilon )\theta.
\end{eqnarray*}
By Condition \ref{cond-dil} and the fact that $\beta_n \Delta_n =
o(n)$, i.e. $\sum_{i=n-\beta_n+1}^n d_{(i)}^{(n)} = o(n)$, we have, for
any $\theta>0$,
$$\lim_{n \rightarrow \infty} \uphi^{(n)}(\theta) = \phi(\theta) , $$
where  $\uphi^{(n)}(\theta)=\EE[e^{-\theta \uD^{(n)}_1}]$. We have for
$\theta>0$,
\begin{eqnarray*}
\PP\left(\uD^{(n)}_1+\dots+\uD^{(n)}_k\leq \eta k\right) \leq
\exp\left(k\left(\theta\eta +\log
    \uphi^{(n)}(\theta) \right)\right) .
\end{eqnarray*}
Fix $\theta<\theta_0$ and let $n$ be sufficiently large so that
$\log \uphi^{(n)}(\theta)\leq \log \phi(\theta)+\epsilon$. This yields
\begin{eqnarray*}
\PP\left(\uD^{(n)}_1+\dots+\uD^{(n)}_k\leq \eta k\right) \leq
\exp\left(k\left(\theta\eta +\log
    \phi(\theta) + \epsilon \theta \right)\right)\\
\leq \exp\left(k\theta\left(\eta -\nu+2 \epsilon \right)\right),
\end{eqnarray*}
which concludes the proof.
\end{proof}

We now prove:
\begin{lemma}\label{lem-R}
For any $\epsilon>0$, we define the event
\begin{eqnarray*}
R''(a):=\left\{ S_k(a) \geq \frac{\nu-1}{1+\epsilon}k,\mbox{ for all } \alpha_n\leq k\leq \beta_n \right\}.
\end{eqnarray*}
For a uniformly chosen vertex $a$, we have $$\PP(R''(a))\geq
1-o(n^{-3/2}).$$
\end{lemma}

\noindent To prove this, we need the following intermediate result
proved in \cite[Theorem 1]{KM2010}:
\begin{lemma}\label{lem-bin}
Let $n_1, n_2 \in \NN$ and $p_1, p_2 \in (0,1)$. We have $\Bin(n_1,p_1) \leq_{st} \Bin(n_2,p_2)$ if and only if the following conditions hold
\begin{itemize}
  \item[$(i)$] $n_1 \leq n_2$,
  \item[$(ii)$] $(1-p_1)^{n_1} \geq (1-p_2)^{n_2}$.
\end{itemize}
\end{lemma}
\noindent In particular, we have
\begin{corollary} \label{cor-bin}
If $x \leq y = o(n)$, we have
$$x - \Bin (x,\sqrt{x/n}) \leq_{st} y - \Bin (y, \sqrt{y/n}) .$$
\end{corollary}
\begin{proof}
By the above lemma, it is sufficient to show
$$(x/n)^{x/2} \geq (y/n)^{y/2} ,$$ and this is true because $s^s$ is decreasing near zero (for $s < e^{-1}$).
\end{proof}
\noindent Now, we go back to the proof of Lemma \ref{lem-R}.

\begin{proof}[Proof of Lemma  \ref{lem-R}]
Note that by Lemma ~\ref{lem-R1} we have
\begin{eqnarray*}
 \PP(i^* \geq \alpha_n) \geq \PP(R'(a)) \geq 1 - o(n^{-3/2}).
\end{eqnarray*}
Then we get
\begin{eqnarray*}
 \PP(R''(a)) \geq 1  - \PP(i^* < \alpha_n) - \PP(R''(a)^c,\: i^* \geq \alpha_n) \\ \geq  1 - o(n^{-3/2}) - \PP(R''(a)^c \mid i^* \geq \alpha_n) .
\end{eqnarray*}
Then to prove Lemma \ref{lem-R} it suffices to prove
$$\PP(R''(a) \mid i^* \geq \alpha_n) \geq 1 - o(n^{-3/2}).$$
By Lemmas \ref{lem-coupl} and \ref{lem-upp}, for any $\epsilon > 0$,
$k \geq \alpha_n$ and $n$ large enough, we have
$$\PP \left( \hd_1 + ... + \hd_k \leq \frac{\nu}{1+\epsilon/2}k \right) \leq e^{-\gamma k} = o(n^{-6}). $$
Then with probability at least $1 - o(n^{-6})$, for any $k\leq \beta_n$,
$$\dmin + \frac{\nu-1}{1+\epsilon/2}k < d_a + \hd_1 + ... + \hd_k - k < (k+1) \Delta_n = o(n) .$$

\noindent By the union bound on $k$, with
probability at least $1 - o(n^{-5})$, we have for all $\alpha_n \leq k \leq \beta_n$,
\begin{eqnarray}\label{ineq:S}
&& \dmin + \frac{\nu-1}{1+\epsilon/2}k < \hS_k(a) < (k+1) \Delta_n = o(n) .
\end{eqnarray}
Then in the remaining of the proof we can assume that the above condition is satisfied.

By Lemma \ref{lem:upX}, Corollary~\ref{cor-bin} and  (\ref{ineq:S}),
conditioning on $\hS_k(a)$ and $\{i^* \geq k\}$, we have:
\begin{eqnarray*}
S_k(a) \geq_{st} \dmin + \frac{\nu-1}{1+\epsilon/2} k - \\ \Bin \left(\dmin + \frac{\nu-1}{1+\epsilon} k, \sqrt{\left(\dmin + \frac{\nu-1}{1+\epsilon} k\right) / n}\right) \\ \geq_{st}
\dmin + \frac{\nu-1}{1+\epsilon/2} k - \Bin \left(\nu k, \sqrt{\nu k /
    n}\right).
\end{eqnarray*}

Corollary~\ref{cor-bin} and  (\ref{ineq:S}) imply that
\begin{eqnarray*}
S_k(a) \mid \{i^* \geq k\} \geq_{st} \dmin + \frac{\nu-1}{1+\epsilon} k - \\ \Bin \left(\dmin + \frac{\nu-1}{1+\epsilon} k, \sqrt{\left(\dmin + \frac{\nu-1}{1+\epsilon} k\right) / n}\right) \\ \geq_{st}
\dmin + \frac{\nu-1}{1+\epsilon} k - \Bin \left(\nu k, \sqrt{\nu k / n}\right) .
\end{eqnarray*}

By Chernoff's inequality, and as $k\sqrt{k/n} = o(k/\sqrt{\alpha_n})$, we have
\begin{eqnarray*}
\PP(\Bin(\nu k, \sqrt{\nu k /n})\geq k/\sqrt{\alpha_n}) &\leq& \exp \left(- \frac{1}{3} k/\sqrt{\alpha_n}\right)  \\
&=& o(n^{-6}) .
\end{eqnarray*}
Moreover, with probability at least $1-o(n^{-6})$, conditioned on $\{i^* \geq k\}$, we have
$$S_k(a)  \geq \dmin + \frac{\nu-1}{1+\epsilon} k - \frac{k}{\sqrt{\alpha_n}} \geq \frac{\nu-1}{1+2\epsilon} k,$$
for $n$ large enough. Then letting $$R''_k = \{ S_k(a) \geq \frac{\nu-1}{1+2\epsilon} k\},$$ we have
\begin{eqnarray}
 \PP(R''_k \mid i^* \geq k ) \geq 1- o(n^{-6}).
\end{eqnarray}

Then $R''(a) = \bigcap_{k=\alpha_n}^{\beta_n} R''_k(a)$, and by using the fact that $R''_{k-1}(a) \subset \{i^* \geq k\}$ we get
\begin{eqnarray*}
\PP(R''(a)) \mid i^* \geq \alpha_n) = 1 - \PP\left(\bigcup_{k=\alpha_n}^{\beta_n} R''_k(a)^c \mid i^* \geq \alpha_n \right) \\
= 1 - \PP\left(R''_{\alpha_n}(a)^c \bigcup_{k=\alpha_n+1}^{\beta_n} \left(R''_k(a)^c \cap R''_{k-1}(a)\right) \mid i^* \geq \alpha_n \right)\\
\geq  1 - \PP\left(R''_{\alpha_n}(a)^c \bigcup_{k=\alpha_n+1}^{\beta_n} \left(R''_k(a)^c \cap  \{i^* \geq k\}\right) \mid i^* \geq \alpha_n \right)\\
\geq 1 -  \sum_{k=\alpha_n}^{\beta_n} \PP(R''_k(a)^c | i^* \geq k)
\geq 1-o(n^{-5}),
\end{eqnarray*}
which concludes the proof.
\end{proof}

\begin{lemma}\label{lem-up-rq}
For a uniformly chosen vertex $a$ and any $\epsilon>0$, we have
\begin{eqnarray*}
\PP\left( T_a(\beta_n) - T_a(\alpha_n) \geq \frac{(1+\epsilon)\log n}{2(\nu-1)} \right) = o(n^{-1}) .
\end{eqnarray*}
\end{lemma}

\begin{proof}
Conditioning on the event $R''(a)$ defined in Lemma~\ref{lem-R}, we have that,
for any $\alpha_n\leq k\leq \beta_n$,
\begin{eqnarray*}
\tau_{k+1} - \tau_k \leq_{st} Y_k \sim \text{Exp}(S_k(a))\leq_{st}\text{Exp} \left(\frac{\nu
    -1}{1+\epsilon} k \right),
\end{eqnarray*}
and all the $Y_k$'s are independent.

Then we let $s = \sqrt{\alpha_n}$, and for $n$ large enough we obtain that

\begin{eqnarray*}
\EE\left[ e^{s (T_a(\beta_n) - T_a(\alpha_n))} \mid R''(a)\right] \leq \prod_{i=\alpha_n}^{\beta_n-1}\left(1 + \frac{s}{\frac{(\nu -1)i}{1+\epsilon} - s} \right) \\
\leq \prod_{i=\alpha_n}^{\beta_n-1}\left(1 + \frac{s(1+2\epsilon)}{(\nu -1)i} \right)\\
\leq \exp\left[\frac{s(1+2\epsilon)}{\nu-1} \sum_{i=\alpha_n}^{\beta_n-1} \frac{1}{i}\right] \\
\leq \exp \left[ \frac {s(1+ 3\epsilon) \log n}{2(\nu -1)}\right].
\end{eqnarray*}
Then we have by Markov's inequality
\begin{eqnarray*}
\PP\left( T_a(\beta_n) - T_a(\alpha_n) \geq \frac{(1+4\epsilon)\log n}{2(\nu-1)} \right) \leq 1 - \PP(R''(a)) \\
  + \PP(R''(a))\EE e^{s(T_a(\beta_n) - T_a(\alpha_n))} \exp\left(- \frac {s(1+ 4\epsilon) \log n}{2(\nu -1)}\right) \\
\leq \exp\left(- \frac{s\epsilon \log n}{2(\nu - 1)}\right)+o(n^{-5})= o(n^{-1}),
\end{eqnarray*}
which concludes the proof.
\end{proof}

Then combining Lemma \ref{lem-up-rq} and \ref{lem-up-r} finishes the proof.

\subsection{Proof of Proposition~\ref{prop-up}.}\label{sec:app-up}
\noindent
Fix two vertices $u$ and $v$. First consider the exploration process for
$B_w(u,t)$ until reaching $t=T_u(\beta_n)$. We know by
Lemma~\ref{lem-R} that, $$S_{\beta_n}(u) \geq (\nu-1-o(1)) \beta_n $$
with probability at least $1-o(n^{-3/2})$.
Thus there are at least $(\nu-1-o(1)) \beta_n$ half-edges in
$B_w(u,T_u(\beta_n))$ except with probability $n^{-3/2}$.

Next, begin exposing $B_w(v,t)$; each matching adds a uniform half-edge to
the neighborhood of $v$. Therefore, the probability that $B_w(v,T_v(\beta_n))$ does not intersect with $B_w(u,T_u(\beta_n))$ is at most
$$\left(1 - \frac{(\nu-1-o(1))\beta_n}{m}\right)^{\beta_n} \leq \exp [-(9-o(1))\log n] < n^{-4}$$
for large $n$ (recall that $\beta_n^2 = 9\frac{\lambda}{\nu-1}n\log n$). The union bound over $u$ and $v$ completes the proof.

\subsection{Proof of Proposition~\ref{prop-low}.}\label{sec:app-low}

We fix a vertex $u \in V_{\dmin}$. Let $\hd'_1, ..., \hd'_{\dmin}$ be the froward degree (i.e. the degree minus one) of neighbors of $u$. Now we consider the exploration process defined in Section~\ref{sec:explor} from the set $N(u)$. Let $\hd'_{\dmin+i}$ be the forward degree of the vertex added at $i$'s exploration step, with $i\geq 1$, and let
\begin{eqnarray}
\hS'_i(u) := \hd'_1 + ... + \hd'_{\dmin + i} - i.
\end{eqnarray}
Again let $\tau_i$ be the time of the $i$'th matching.
We have
$$\tau_{i+1}-\tau_i \geq_{st} Y_i \sim \text{Exp}\left( \hS'_i(u) \right),$$
and all the $Y_i$'s are independent. This follows from the fact that the worst case is when the explored set forms a tree. Also by Lemma \ref{lem-coupl}, we have
$$\sum_{j=1}^{\dmin+i} \hd'_j \leq_{st}  \sum_{j=1}^{\dmin+i} \bD^{(n)}_j ,$$
where $\bD^{(n)}_j$ are i.i.d with distribution $\bpi^{(n)}$. Let $\bnu^{(n)}$ be the expected value of $\bD^{(n)}_1$ which is:
$$\bnu^{(n)} := \sum_k k \bpi_k^{(n)} ,$$ and let
$z_n = \sqrt{n / \log n } $. Now we show that $\tau_{z_n} \geq t_n$ with high probability.

Let us define  $$T'(k) \sim \sum_{i=1}^k  \text{Exp}\left( \sum_{j=1}^{\dmin+i} \bD^{(n)}_j - i \right) ,$$
where all the exponential variables are independents. Then we have $\tau_{z_n} \geq_{st} T'(z_n)$.

\begin{lemma}\label{lem-coupl-exp}
Let $X_1,...,X_t$ be a random process adapted to a filtration $\mathcal{F}_0 = \sigma[\o], \mathcal{F}_1,...,\mathcal{F}_t$, and let $\mu_i = \EE X_i$, ${\bf \Sigma}_i = X_1+ ...+X_i$, $\Lambda_i = \mu_1 + ... + \mu_i$. Let $Y_i \sim \text{Exp}(\Sigma_i)$, and $Z_i \sim \text{Exp}(\Lambda_i)$, where all exponential variables are independents. Then we have
$$Y_1 + ... + Y_t \geq_{st} Z_1 + ... + Z_t .$$
\end{lemma}
\begin{proof}
By Jensen's inequality it is easy to see that for positive random variable $X$, we have
$$\text{Exp}(X) \geq_{st} \text{Exp}(\EE X) . $$
Then by induction, it suffices to prove that for a pair of random variables $X_1$, $X_2$ we have $Y_1 + Y_2 \geq_{st} Z_1 + Z_2 $. We have
\begin{eqnarray*}
\PP(Y_1 + Y_2 > s) = \EE_{X_1} [\PP(Y_1 + Y_2 > s | X_1 )] \\
\geq \EE_{X_1} [\PP(\text{Exp}(X_1) + \text{Exp}(X_1 + \mu_2) > s)] \\
\geq \PP(Z_1 + Z_2 >s) .
\end{eqnarray*}
\end{proof}
Then by Lemma~\ref{lem-coupl-exp}, we have
$$T'(z_n) \geq_{st} T^*(z_n) := \sum_{i=0}^{z_n} \text{Exp}\left(\dmin \bnu^{(n)} + (\bnu^{(n)} - 1) i\right), $$ where all exponential variables are independents.
Let $b := \dmin \bnu^{(n)} - (\bnu^{(n)} - 1)$. Then similarly to \cite{ding09}, we have
\begin{eqnarray*}
\PP(T^*(z_n) \leq t) \leq \int_{\sum x_i \leq t}e^{- \sum_{i=1}^{z_n} ((\bnu^{(n)} - 1)i + b)x_i}  dx_1 ... dx_{z_n}  \\
 \prod_{i=1}^{z_n} ((\bnu^{(n)} - 1)i + b) \ \ \ \ \ \\
= \int_{0 \leq y_1 \leq ... \leq t}e^ {-(\bnu^{(n)} - 1) \sum_{i=1}^{z_n} y_i} e^{-by_{z_n}} dy_1 ... dy_{z_n}  \\
\prod_{i=1}^{z_n} ((\bnu^{(n)} - 1)i + b), \ \ \ \ \
\end{eqnarray*}
where $y_k = \sum_{i=0}^{k-1} x_{z_n-i}$.
Letting $y$ play the role of $y_{z_n}$ and accounting for all
permutations over $y_1, ..., y_{z_{n}-1}$ (giving each such variable the range $[0, y]$),
\begin{eqnarray*}
\PP(T^*(z_n) \leq t) \leq \int_0^t  e^{-(\bnu^{(n)} - 1 + b) y} dy \frac{\prod_{i=1}^{z_n} (i + \frac{b}{\bnu^{(n)} - 1})}{(z_n-1)!}\\ . \int_{[0,y]^{z_n-1}} (\bnu^{(n)} - 1)^{z_n} e^{-(\bnu^{(n)} - 1)\sum_{i=1}^{z_n-1}y_i} dy_1...dy_{z_n-1}  \\
\leq \int_0^t  e^{-(\bnu^{(n)} - 1 + b) y} dy \frac{\prod_{i=1}^{z_n} (i + \frac{b}{\bnu^{(n)} - 1})}{(z_n-1)!}\\
\prod_{i=1}^{z_n-1} \int_0^y (\bnu^{(n)} - 1) e^{-(\bnu^{(n)} - 1)y_i} dy_i \\
\leq \int_0^t  e^{-\dmin \bnu^{(n)}} (1 - e^{-(\bnu^{(n)} - 1)y})^{z_n-1} dy\\
c (\bnu^{(n)} - 1) z_n^{\frac{b}{\bnu^{(n)} - 1}+1},\ \ \ \ \
\end{eqnarray*}
where $c>0$ is an absolute constant. Then we obtain
$$\PP(T^*(z_n) \leq t_n) \leq  c (\bnu^{(n)} - 1) z_n^{\frac{b}{\bnu^{(n)} - 1}} \int_0^t e^{-n^\epsilon} dy = o(n^{-4}) .$$

Then w.h.p. we have $|B'_w(u,t_n)| \leq z_n$. Choosing another vertex $v$, at random, and exposing $B'_w(v,t_n)$, again w.h.p we obtain a set of size at most $z_n$. Now because each matching is uniform among the remaining half-edges, then its probability of hitting $B'_w(u,t_n)$ is at most $\hS'_{z_n}(u)/n$.

Let $\epsilon_n := \log \log n$. By Markov's inequality we have
\begin{eqnarray*}
 \PP\left(\hS'_{z_n}(u) \geq z_n \epsilon_n \right) &\leq& \EE \hS'_{z_n}(u)/z_n \epsilon_n \\
&\leq& \frac{\dmin \bnu^{(n)}  + (\bnu^{(n)} - 1) z_n}{z_n \epsilon_n} = o(1).
\end{eqnarray*}
We conclude
$$\PP(B'_w(u,t_n)(u) \cap B'_w(v,t_n) \neq \emptyset) \leq \epsilon_n z_n^2/n = o(1) ,$$
which completes the proof.

\subsection{Proof of Lemma~\ref{lem:connect}.}\label{se:app-connect}

By Lemma~\ref{lem-R1}, $R'(a)$ holds with probability at leat $1-o(n^{-1})$. Then with probability $1-o(n^{-1})$, for an uniformly chosen vertex $a$, we have $S_k(a) \geq 1$ for all $1 \leq k \leq \alpha_n$. Then by union bound with probability $1-o(1)$, for all nodes $a \in V$, the size of the cluster $C_a$, starting from $a$ reaches $\alpha_n$.
Then we use Lemma~\ref{lem-R} to show that for all nodes, this cluster also reaches $\beta_n$. Now it is easy to conclude by Proposition~\ref{prop-up}.

\fi

\end{document}